\documentclass[12pt]{article}
\usepackage{fullpage}
\usepackage{amssymb}
\usepackage{amsmath}
\usepackage{amsthm}
\usepackage{stmaryrd}
\usepackage{xypic}
\usepackage{setspace}
\usepackage{array}
\usepackage{rotating}

\newtheorem{theorem}{Theorem}[section]
\newtheorem{lemma}[theorem]{Lemma}
\newtheorem{proposition}[theorem]{Proposition}
\newtheorem{corollary}[theorem]{Corollary}

\numberwithin{equation}{section}

\newcommand{\tx}[1]{\textnormal{#1}}
\newcommand{\gp}[2]{\langle #1 \mid #2 \rangle}
\newcommand{\s}{\sigma}
\newcommand{\G}{\Gamma}
\newcommand{\calK}{\mathcal K}
\newcommand{\calP}{\mathcal P}
\newcommand{\calQ}{\mathcal Q}

\newcommand{\calR}{\mathcal R}

\newcommand{\GP}{\Gamma^+(\calP)}
\newcommand{\ch}[1]{\overline{#1}}
\newcommand{\GcP}{\Gamma^+(\ch{\calP})}
\newcommand{\GcQ}{\Gamma^+(\ch{\calQ})}
\newcommand{\GQ}{\Gamma^+(\calQ)}
\newcommand{\GR}{\Gamma^+(\calR)}
\newcommand{\GK}{\Gamma^+(\calK)}
\newcommand{\GPQ}{\G^+(\calP \mix \calQ)}
\newcommand{\GcPQ}{\GP \comix \GQ}
\newcommand{\mix}{\diamond}
\newcommand{\comix}{\boxempty}
\newcommand{\eps}{\epsilon}

\newcommand{\tref}[1]{Theorem~\textup{\ref{#1}}}
\newcommand{\pref}[1]{Proposition~\textup{\ref{#1}}}
\newcommand{\cref}[1]{Corollary~\textup{\ref{#1}}}
\newcommand{\lref}[1]{Lemma~\textup{\ref{#1}}}
\newcommand{\eref}[1]{Equation~\textup{\ref{#1}}}

\begin{document}

\title{Mixing Chiral Polytopes}

\author{Gabe Cunningham\\
Department of Mathematics\\
Northeastern University\\
Boston, Massachusetts,  USA, 02115
}

\date{ \today }
\maketitle

\begin{abstract}
An abstract polytope of rank $n$ is said to be \emph{chiral} if its automorphism group has two orbits on 
the flags, such that adjacent flags belong to distinct orbits. Examples of chiral polytopes have been difficult
to find. A ``mixing'' construction lets us combine polytopes to build new regular and chiral polytopes. By using the chirality 
group of a polytope, we are able to give simple criteria for when the mix of two polytopes is chiral.



\end{abstract}

\section{Introduction}

	The study of abstract polytopes is a growing field, uniting combinatorics with geometry and
	group theory. At the forefront are the (abstract) regular polytopes, including the regular convex
	polytopes, regular tessellations of space-forms, and many new combinatorial structures with
	maximal symmetry. Recently, the study of chiral polytopes has flourished. Chiral polytopes are
	``half-regular''; the action of the automorphism group on the flags has two orbits, and adjacent
	flags belong to distinct orbits. A chiral polytope occurs in two enantiomorphic (mirror-image) forms, and 
	though these forms are isomorphic as polytopes, the particular orientation chosen is usually relevant.
	
	Chiral maps (also called irreflexible maps) have been studied for some time (see \cite{cm}), and the study of 
	chiral maps and hypermaps continues to yield interesting developments (for example, see \cite{aj}).
	However, it was only with the introduction of abstract polytopes that the notion of chirality was defined for 
	polytopes in ranks 4 and higher \cite{SW1}. 
	
	Examples of chiral polytopes have been hard to find. A few families have been found in ranks
	$3$ and $4$, but only a handful of examples are known in ranks $5$ and higher.
	There are several impediments
	to constructing new examples. Foremost is that the $(n-2)$-faces and the co-faces at the edges 
	of a chiral polytope must be regular. Therefore, it is not possible to repeatedly
	extend a chiral polytope to higher and higher ranks -- we need genuinely new examples in each rank.
	So far, nobody has found a ``nice'' family of chiral polytopes of arbitrary rank. Indeed, it
	was only recently that Pellicer demonstrated conclusively for the first time that there are
	chiral polytopes in every rank \cite{pel} .
	
	Since it is so difficult to extend chiral polytopes to higher ranks, we need another tactic for building
	new chiral polytopes. In \cite{const}, the authors adapted the mixing technique used in \cite{arp} to chiral 
	polytopes. Two main difficulties arise. The first is that the mix of two polytopes is not necessarily 
	polytopal: specifically, the resulting group does not necessarily have the required intersection property. 
	However, under some fairly mild conditions on the polytopes being mixed, we can ensure that the mix
	is, in fact, polytopal. The second difficulty is that we need a way of determining whether
	the mix is chiral or not, which can be difficult to do directly. We would
	like simple combinatorial criteria which will tell us when the mix is chiral.

	By using the idea of the \emph{chirality group} of a polytope, introduced in \cite{bjns} for hypermaps and in 
	\cite{const} for polytopes, we are able to outline such criteria. We then apply these results to construct
	new examples of chiral $5$-polytopes. We also see how to construct infinitely many chiral $n$-polytopes
	given a single chiral $n$-polytope satisfying some mild conditions.
	
	We start by giving some background information on regular and chiral abstract polytopes in Section 2. 
	In Section 3, we introduce the mixing operation for chiral and directly regular polytopes,
	and we give a few results (including one new one) for when the mix of two polytopes is again a polytope.
	In Section 4, we define the chirality group of a polytope, which we then use to give several simple criteria
	for when the mix of two or more polytopes is chiral. Finally, in Section 5 we highlight
	the main results with several constructions that build new chiral polytopes.

\section{Polytopes}

	General background information on abstract polytopes can be found in \cite[Chs. 2, 3]{arp}, and information
	on chiral polytopes specifically can be found in \cite{SW1}.
	Here we review the concepts essential for this paper.

	\subsection{Definition of a polytope}

		Let $\calP$ be a ranked partially ordered set whose elements will be called \emph{faces}. 
		The faces of $\calP$ will range in rank from $-1$ to $n$, and a face of rank $j$ is called a 
		\emph{$j$-face}. The $0$-faces, $1$-faces, and $(n-1)$-faces are also 
		called \emph{vertices}, \emph{edges}, and \emph{facets}, respectively. A \emph{flag} of
		$\calP$ is a maximal chain. We say that two flags are adjacent ($j$-adjacent) if they differ in exactly
		one face (their $j$-face, respectively). If $F$ and $G$ are faces of $\calP$
		such that $F < G$, then the \emph{section} $G / F$ consists of those faces $H$ such that
		$F \leq H \leq G$.
		
		We say that $\calP$ is an \emph{(abstract) polytope of rank $n$}, also called an \emph{$n$-polytope}, 
		if it satisfies the following four properties:
		
		\begin{enumerate}
		\item There is a unique greatest face $F_n$ of rank $n$ and a unique least face $F_{-1}$ of rank $-1$.
		\item Each flag of $\calP$ has $n+2$ faces.
		\item $\calP$ is \emph{strongly flag-connected}, meaning that if $\Phi$ and $\Psi$
		are two flags of $\calP$, then there is a sequence of flags $\Phi = \Phi_0, \Phi_1, \ldots, \Phi_k = \Psi$ 
		such that for $i = 0, \ldots, k-1$, the flags $\Phi_i$ and $\Phi_{i+1}$ are adjacent, and each
		$\Phi_i$ contains $\Phi \cap \Psi$.
		\item (Diamond condition): Whenever $F < G$, where $F$ is a $(j-1)$-face and $G$ is a $(j+1)$-face
		for some $j$, then there are exactly two $j$-faces $H$ with $F < H < G$.
		\end{enumerate}
		
		If $F$ is a $j$-face and $G$ is a $k$-face of a polytope with $F \leq G$, then the section $G/F$ is a
		($k-j-1$)-polytope itself. We can identify a face $F$ with the section $F/F_{-1}$; if $F$ is a $j$-face,
		then $F/F_{-1}$ is a $j$-polytope. We call the section $F_n/F$ the \emph{co-face at $F$}. The co-face
		at a vertex is also called a \emph{vertex-figure}.

		We sometimes need to work with \emph{pre-polytopes}, which are ranked partially ordered sets that
		satisfy the first, second, and fourth property above, but not necessarily the third. In this paper, all 
		of the pre-polytopes we encounter will be \emph{flag-connected}, meaning that if $\Phi$ and $\Psi$ are two 
		flags, there is a sequence of flags $\Phi = \Phi_0, \Phi_1, \ldots, \Phi_k = \Psi$ such
		that for $i = 0, \ldots, k-1$, the flags $\Phi_i$ and $\Phi_{i+1}$ are adjacent (but we do not require each
		flag to contain $\Phi \cap \Psi$). When working with pre-polytopes, we apply all the same terminology as 
		with polytopes. 
	
	\subsection{Regularity}
	
		For polytopes $\calP$ and $\calQ$, an \emph{isomorphism} from $\calP$ to $\calQ$ is an incidence- and rank-preserving bijection on the set of 
		faces. An isomorphism from $\calP$ to itself is an \emph{automorphism} of $\calP$. We denote the group of 
		all automorphisms of $\calP$ by $\G(\calP)$. There is a natural action of $\G(\calP)$ on the
		flags of $\calP$, and we say that $\calP$ is \emph{regular} if this action is transitive.
		This coincides with any of the usual definitions of regularity for convex polytopes.
		
		Given a regular polytope $\calP$, fix a \emph{base flag} $\Phi$. Then the automorphism
		group $\G(\calP)$ is generated by involutions $\rho_0, \ldots, \rho_{n-1}$
		where $\rho_i$ maps $\Phi$ to the flag $\Phi^i$ that is $i$-adjacent to $\Phi$. 
		These generators satisfy $(\rho_i \rho_j)^2 = \eps$ for all $i$ and $j$ such that $|i - j| \geq 2$.
		We say that $\calP$ has (\emph{Schl\"afli}) \emph{type} $\{p_1,\ldots,p_{n-1}\}$ if
		for each $i = 1, \ldots, n-1$ the order of $(\rho_{i-1} \rho_i)$ is $p_i$ (with $2 \leq p_i \leq \infty$).
		We also use $\{p_1, \ldots, p_{n-1}\}$ to represent the universal regular polytope of this type,
		which has an automorphism group with no relations other than those mentioned above.
		We denote $\G(\{p_1, \ldots, p_{n-1}\})$ by $[p_1, \ldots, p_{n-1}]$.
		Whenever this universal polytope corresponds to a regular convex polytope, then the name used
		here is the same as the usual Schl\"afli symbol for that polytope (see \cite{coxeter}).
		
		For $I \subseteq \{0, 1, \ldots, n-1\}$ and a group $\G = \langle \rho_0, \ldots, \rho_{n-1} \rangle$,
		we define $\G_I := \langle \rho_i \mid i \in I \rangle$. 
		The strong flag-connectivity of polytopes induces the following {\em intersection property\/} in the group:
		\begin{equation}
		\label{eq:reg-int}
		\G_I \cap \G_J = \G_{I \cap J}.
		\end{equation}

		In general, if $\G = \langle \rho_0, \ldots, \rho_{n-1} \rangle$ is a group such that each
		$\rho_i$ has order $2$ and such that $(\rho_i \rho_j)^2 = \eps$ whenever $|i - j| \geq 2$, then
		we say that $\G$ is a \emph{string group generated by involutions} (or \emph{sggi}). If
		$\G$ also satisfies the intersection property given above, then we call $\G$ a \emph{string
		C-group}. There is a natural way of building a regular polytope $\calP(\G)$ from a string
		C-group $\G$ such that $\G(\calP(\G)) = \G$ (see \cite[Ch. 2E]{arp}). Therefore, we get a one-to-one 
		correspondence between regular $n$-polytopes and string C-groups on $n$ specified generators.

	\subsection{Direct Regularity and Chirality}

		If $\calP$ is a regular polytope with automorphism group $\G(\calP)$ generated
		by $\rho_0, \ldots, \rho_{n-1}$, then the elements
		\[ \s_i := \rho_{i-1} \rho_i \]
		(for $i=1, \ldots, n-1$) generate the \emph{rotation subgroup} $\GP$ of $\G(\calP)$, which has index at 
		most~$2$. We say that $\calP$ is \emph{directly regular} if this index is $2$. This is essentially an
		orientability condition; for example, the directly regular polyhedra correspond to orientable maps.
		The convex regular polytopes are all directly regular.
	
		We say that an $n$-polytope $\calP$ is \emph{chiral} if the action of $\G(\calP)$ on the flags
		of $\calP$ has two orbits such that adjacent flags are always in distinct orbits.
		For convenience, we define $\GP = \G(\calP)$ whenever $\calP$ is chiral.
		Given a chiral polytope $\calP$, fix a base flag $\Phi=\{F_{-1}, F_0, \ldots, F_n\}$.
		Then the automorphism group $\GP$ is generated by elements $\sigma_1, \ldots, \sigma_{n-1}$,
		where $\s_i$ acts on $\Phi$ the same way that $\rho_{i-1} \rho_i$ acts on the base flag of a regular
		polytope. That is, $\s_i$ sends $\Phi$ to $\Phi^{i,i-1}$. For $i < j$, we get that $(\s_i \cdots
		\s_j)^2 = \eps$. In analogy to regular polytopes, if the order of each $\s_i$ is $p_i$, 
		we say that the $\emph{type}$ of $\calP$ is $\{p_1, \ldots, p_{n-1}\}$.

		The automorphism groups of chiral polytopes and the rotation groups of directly regular polytopes satisfy 
		an intersection property analogous to that for string C-groups. Let $\G^+ = \langle \s_1, \ldots, \s_{n-1}
		\rangle$ be the rotation group of a
		chiral or directly regular polytope. For $1 \leq i \leq j \leq n-1$, we define 
		\[ \tau_{i,j}:= \s_i \s_{i+1} \cdots \s_j, \]
		and for $0 \leq i \leq n$ we let $\tau_{0,i}:=\tau_{i,n}:= \eps$.
		For $I \subseteq \{-1, 0, \ldots, n\}$ we define
		\[ \G^+_I := \langle\tau_{i,j}\mid i\leq j \mbox{ and } i-1,j\in I\rangle.  \]
		Then the \emph{intersection property} for $\G^+$ is given by:
		\begin{equation}
		\label{eq:chiral-int}
		\G^+_I \cap \G^+_J = \G^+_{I\cap J}  
		\end{equation}
	
		If $\G^+$ is a group generated by elements $\s_1, \ldots, \s_{n-1}$ such that $(\s_i \cdots \s_j)^2 =
		\eps$ for $i < j$, and if $\G^+$ satisfies the intersection property above,
		then $\G^+$ is either the automorphism group of a chiral $n$-polytope or the rotation subgroup 
		of a directly regular polytope. In particular, it is the rotation subgroup
		of a directly regular polytope if and only if there is an automorphism of $\G^+$ that sends $\s_1$ to 
		$\s_1^{-1}$, $\s_2$ to $\s_1^2 \s_2$, and fixes every other generator.
		
		Suppose $\calP$ is a chiral polytope with base flag $\Phi$ and with $\GP = \langle \s_1, \ldots,
		\s_{n-1} \rangle$. Let $\ch{\calP}$ be the chiral polytope with the same underlying face-set
		as $\calP$, but with base flag $\Phi^0$. Then $\G^+(\ch{\calP}) = \langle \s_1^{-1}, \s_1^2 \s_2, \s_3,
		\ldots, \s_{n-1} \rangle$. We call $\ch{\calP}$ the \emph{enantiomorphic form} or \emph{mirror image}
		of $\calP$. Though $\calP \simeq \ch{\calP}$, there is no automorphism of $\calP$ that takes
		$\Phi$ to $\Phi^0$.
		
		Let $\G^+ = \langle \s_1, \ldots, \s_{n-1} \rangle$, and let $w$ be a word in the free group on these 
		generators. We define the \emph{enantiomorphic} (or \emph{mirror image}) word $\ch{w}$ of $w$ to be the word obtained from $w$ by replacing every occurrence of $\s_1$ by $\s_1^{-1}$ and $\s_2$ by $\s_1^2\s_2$, 
		while keeping all $\s_j$ with $j\geq 3$ unchanged. Then if $\G^+$ is the rotation subgroup of a directly
		regular polytope, the elements of $\G^+$ corresponding to $w$ and $\ch{w}$ are conjugate in $\G$.
		On the other hand, if $\G^+$ is the automorphism group of a chiral polytope, then $w$ and $\ch{w}$ need 
		not even have the same period. Note that $\ch{\ch{w}} = w$ for all words $w$.
		
		The sections of a regular polytope are again regular, and the sections of a chiral polytope are either
		directly regular or chiral. Furthermore, for a chiral $n$-polytope, all the $(n-2)$-faces and all the 
		co-faces at edges must be directly regular. As a consequence, if $\calP$ is a chiral polytope, it may be
		possible to extend it to a chiral polytope having facets isomorphic to $\calP$, but it will then
		be impossible to extend that polytope once more to a chiral polytope. 

		Chiral polytopes only exist in ranks 3 and higher.
		The simplest examples of are the toroidal maps $\{4,4\}_{(b,c)}$, $\{3,6\}_{(b,c)}$ and 
		$\{6,3\}_{(b,c)}$, with $b,c\neq 0$ and $b\neq c$ (see \cite{cm}). These give rise to chiral $4$-polytopes
		having toroidal maps as facets and/or vertex-figures. More examples of chiral 4- and 5-polytopes can
		be found in \cite{chp}.
	
		If a regular or chiral $n$-polytope $\calP$ has facets $\calP_1$ and vertex-figures $\calP_2$, we say 
		that $\calP$ is of \emph{type} $\{\calP_1,\calP_2\}$. Given regular or chiral polytopes $\calP_1$ and 
		$\calP_2$, if there are any regular or chiral polytopes of type $\{\calP_1, \calP_2\}$, then there is a 
		universal one that covers all other regular or chiral polytopes of that type. We then also use $\{\calP_1, 
		\calP_2\}$ to denote this universal polytope.

		Let $\calP$ and $\calQ$ be two polytopes (or flag-connected pre-polytopes) of the same rank, not 
		necessarily regular or chiral. A mapping $\gamma: \calP \to \calQ$ is called a
		\emph{covering} if it preserves incidence of faces, ranks of faces, and adjacency of flags; then $\gamma$ is
		necessarily surjective, by the flag-connectedness of $\calQ$. We say that $\calP$ \emph{covers} $\calQ$
		if there exists a covering $\gamma: \calP \to \calQ$.
		
		If $\calP$ and $\calQ$ are chiral or directly regular $n$-polytopes, their rotation groups
		are both quotients of 
		\[ W^+ := [\infty, \ldots, \infty]^+ = \langle \s_1, \ldots, \s_{n-1} \mid (\s_i \cdots \s_j)^2
		= \eps \tx{ for $i < j$} \rangle. \]
		Therefore there are normal subgroups $M$ and $K$ of $W^+$ such that $\GP = W^+/M$ and $\GQ = W^+/K$. Then
		$\calP$ covers $\calQ$ if and only if $M \leq K$.
	 
		Let $\calP$ be a chiral or directly regular polytope with $\GP = W^+/M$. We define 
		\[ \ch{M} = \{\ch{w} \mid w \in M\}. \]
		If $\ch{M} = M$, then $\calP$ is directly regular. Otherwise, $\calP$ is chiral, and $\GcP = W^+/\ch{M}$.

\section{Mixing polytopes}

	In this section, we will define the mix of two finitely presented groups, which naturally
	gives rise to a way to mix polytopes. The mixing operation is analogous to the join of hypermaps 
	\cite{ant2} and the parallel product of maps \cite{wilson}.
	
	Let $\G = \langle x_1, \ldots, x_n \rangle$ and $\G' =
	\langle x_1', \ldots, x_n' \rangle$ be groups with $n$ specified generators. Then the elements
	$z_i = (x_i, x_i') \in \G \times \G'$ (for $i = 1, \ldots, n$) generate a subgroup of
	$\G \times \G'$ that we call the \emph{mix} of $\G$ and $\G'$ and denote $\G \mix \G'$
	(see \cite[Ch.7A]{arp}).

	If $\calP$ and $\calQ$ are chiral or directly regular $n$-polytopes, we can mix their rotation groups.
	Let $\GP = \langle \s_1, \ldots, \s_{n-1} \rangle$ and $\GQ = \langle
	\s_1', \ldots, \s_{n-1}' \rangle$. Let $\beta_i = (\s_i, \s_i')$ for $i = 1, \ldots, n-1$.
	Then $\GP \mix \GQ = \langle \beta_1, \ldots, \beta_{n-1} \rangle$. We note that for $i < j$, we have
	$(\beta_i \cdots \beta_j)^2 = \eps$, so that the group $\GP \mix \GQ$ can be written as a quotient of $W^+$.
	In general, however, it will not have the intersection property (\eref{eq:chiral-int}) with respect to its 
	generators $\beta_1, \ldots, \beta_{n-1}$. Nevertheless, it is possible to build a directly regular or
	chiral poset from
	$\GP \mix \GQ$ using the method outlined in \cite{SW1}, and we denote that poset $\calP \mix \calQ$
	and call it the \emph{mix} of $\calP$ and $\calQ$. (In fact, this poset is always a flag-connected
	pre-polytope.) Thus $\GPQ = \GP \mix \GQ$. If $\GP \mix \GQ$ satisfies
	the intersection property, then $\calP \mix \calQ$ is in fact a polytope.
	
	The following proposition is proved in \cite{const}:

	\begin{proposition}
	\label{prop:mix}
	Let $\calP$ and $\calQ$ be chiral or directly regular polytopes with $\GP = W^+/M$ and
	$\GQ = W^+/K$. Then $\GPQ \simeq W^+/(M \cap K)$.
	\end{proposition}
	
	Determining the size of $\GP \mix \GQ$ is often difficult for a computer unless $\GP$ and $\GQ$ are
	both fairly small. However, there is usually an easy way to indirectly calculate the size
	using the \emph{comix} of two groups. If $\G$ has presentation $\gp{x_1, \ldots, x_n}{R}$ and $\G'$ has 
	presentation $\gp{x_1', \ldots, x_n'}{S}$,
	then we define the comix of $\G$ and $\G'$, denoted $\G \comix \G'$, to be the group with presentation
	\[ \gp{x_1, x_1', \ldots, x_n, x_n'}{R, S, x_1^{-1}x_1', \ldots, x_n^{-1}x_n'}.\]
	Informally speaking, we can just add the relations from $\G'$ to $\G$, rewriting them to use
	$x_i$ in place of $x_i'$. 

	Just as the mix of two rotation groups has a simple description in terms of quotients of $W^+$, so
	does the comix of two rotation groups:

	\begin{proposition}
	\label{prop:comix}
	Let $\calP$ and $\calQ$ be chiral or directly regular polytopes with $\GP = W^+/M$ and
	$\GQ = W^+/K$. Then $\GcPQ \simeq W^+/MK$.
	\end{proposition}

	\begin{proof}
	Let $\GP = \langle \s_1, \ldots, \s_{n-1} \mid R \rangle$, and let $\GQ = \langle \s_1, \ldots, \s_{n-1} \mid
	S \rangle$, where $R$ and $S$ are sets of relators in $W^+$. 
	Then $M$ is the normal closure of $R$ in $W^+$ and $K$ is the normal closure of $S$ in $W^+$.
	We can write $\GcPQ = \langle \s_1, \ldots, \s_{n-1} \mid R \cup S \rangle$, so we want to show that $MK$ is
	the normal closure of $R \cup S$ in $W^+$. It's clear that $MK$ contains $R \cup S$, and since
	$M$ and $K$ are normal, $MK$ is normal, and so it contains the normal closure of $R \cup S$.
	To show that $MK$ is contained in the normal closure of $R \cup S$, it suffices to show that if
	$N$ is a normal subgroup of $W^+$ that contains $R \cup S$, then it must also contain $MK$. Clearly,
	such an $N$ must contain the normal closure $M$ of $R$ and the normal closure $K$ of $S$. Therefore,
	$N$ contains $MK$, as desired.
	\end{proof}

	Now we can determine how the size of $\GP \mix \GQ$ is related to the size of $\GP \comix \GQ$.
	
	\begin{proposition}
	\label{prop:SizeOfMix}
	Let $\calP$ and $\calQ$ be finite chiral or directly regular $n$-polytopes. Then 
	\[ |\GPQ| \cdot |\GcPQ| = |\GP| \cdot |\GQ|. \]
	\end{proposition}

	\begin{proof}
	Let $\GP = W^+/M$ and $\GQ = W^+/K$. Then by \pref{prop:mix}, $\GPQ = W^+/(M \cap K)$, and 
	by \pref{prop:comix}, $\GcPQ = W^+/MK$.
	Let $\pi_1: \GPQ \to \GP$ and $\pi_2: \GQ \to \GcPQ$ be the natural covering maps. Then
	$\ker \pi_1 \simeq M/(M \cap K)$ and $\ker \pi_2 \simeq MK/K \simeq M/(M \cap K)$. Therefore,
	we have that 
	\[|\GPQ| = |\GP||\ker \pi_1| = |\GP||\ker \pi_2| = |\GP||\GQ|/|\GcPQ|,\]
	and the result follows.
	\end{proof}

	\begin{corollary}
	\label{cor:trivial-comix}
	Let $\calP$ and $\calQ$ be finite chiral or directly regular $n$-polytopes such that $\GP \comix \GQ$
	is trivial. Then $\GP \mix \GQ = \GP \times \GQ$.
	\end{corollary}
	
	The reason that \pref{prop:SizeOfMix} is so useful in calculating the size of $\GP \mix \GQ$ is that
	it is typically very easy for a computer to find the size of $\GP \comix \GQ$. Indeed, in many of the
	cases that come up in practice, it is easy to calculate $|\GP \comix \GQ|$ by hand just by combining
	the relations from $\GP$ and $\GQ$ and playing with the presentation a little. 
	
	The mix of $\calP$ and $\calQ$ is polytopal if and only if $\GP \mix \GQ$ satisfies the intersection
	property (\eref{eq:chiral-int}). There is no general method for determining whether this condition
	is met, but the following two results from \cite{const} are widely applicable.
		
	\begin{proposition}
	\label{prop:facets-cover}
	Let $\calP$ and $\calQ$ be chiral or directly regular $n$-polytopes. If the facets of $\calP$ cover
	the facets of $\calQ$, or if the vertex-figures of $\calP$ cover the vertex-figures of $\calQ$,
	then $\calP \mix \calQ$ is polytopal.
	\end{proposition}

	\begin{proposition}
	\label{prop:rel-prime-type}
	Let $\calP$ be a chiral or directly regular $n$-polytope of type $\{p_1, \ldots, p_{n-1}\}$, and let
	$\calQ$ be a chiral or directly regular $n$-polytope of type $\{q_1, \ldots, q_{n-1}\}$.
	If $p_i$ and $q_i$ are relatively prime for each $i = 1, \ldots, n-1$, then $\calP \mix \calQ$ is a 
	chiral or directly regular $n$-polytope of type $\{p_1 q_1, \ldots, p_{n-1} q_{n-1}\}$, and $\GPQ = 
	\GP \times \GQ$.
	\end{proposition}
	
	In fact, it is actually sufficient for only the middle entries of the Schl\"afli symbol to be coprime:

	\begin{theorem}
	\label{thm:rel-prime-type}
	Let $\calP$ be a chiral or directly regular $n$-polytope of type $\{p_1, \ldots, p_{n-1}\}$, and let
	$\calQ$ be a chiral or directly regular $n$-polytope of type $\{q_1, \ldots, q_{n-1}\}$.
	If $p_i$ and $q_i$ are relatively prime for each $i = 2, \ldots, n-2$ (but not neccessarily for
	$i=1$ or $i=n-1$), then $\calP \mix \calQ$ is a chiral or directly regular $n$-polytope.
	\end{theorem}

	\begin{proof}
	We prove the claim by induction. The claim is trivially true for $n \leq 2$. Now, suppose the claim is true 
	for $(n-1)$-polytopes, and let $\calP$ and $\calQ$ be $n$-polytopes satisfying the given conditions. 
	Let $\GP = \langle \s_1, \ldots, \s_{n-1} \rangle$ and $\GQ = \langle \s_1', \ldots, \s_{n-1}' \rangle$.
	Let $\beta_i = (\s_i, \s_i')$, so that $\GP \mix \GQ = \langle \beta_1, \ldots, \beta_{n-1} \rangle$.
	Now, the facets of $\calP$ are of type $\{p_1, \ldots, p_{n-2}\}$ and the facets of $\calQ$ are of type
	$\{q_1, \ldots, q_{n-2}\}$, so that by the inductive hypothesis, the mix of the facets is polytopal.
	In other words, $\langle \beta_1, \ldots, \beta_{n-2} \rangle$ has the intersection property
	(\eref{eq:chiral-int}). Similarly, $\langle \beta_2, \ldots, \beta_{n-1} \rangle$ has the intersection
	property. Now, if we can prove that $\langle \beta_1, \ldots, \beta_{n-2} \rangle \cap \langle
	\beta_2, \ldots, \beta_{n-1} \rangle \leq \langle \beta_2, \ldots, \beta_{n-2} \rangle$, then it
	follows that $\calP \mix \calQ$ is polytopal \cite{SW1}. We have that
	\begin{align*}
	\langle \beta_1, \ldots, \beta_{n-2} \rangle & \cap \langle \beta_2, \ldots, \beta_{n-1} \rangle \\
	&= \langle (\s_1, \s_1'), \ldots, (\s_{n-2}, \s_{n-2}') \rangle \cap 
	\langle (\s_2, \s_2'), \ldots, (\s_{n-1}, \s_{n-1}') \rangle \\
	& \leq (\langle \s_1, \ldots, \s_{n-2} \rangle \times \langle \s_1', \ldots, \s_{n-2}' \rangle) \cap
	(\langle \s_2, \ldots, \s_{n-1} \rangle \times \langle \s_2', \ldots, \s_{n-1}' \rangle) \\
	&= (\langle \s_1, \ldots, \s_{n-2} \rangle \cap \langle \s_2, \ldots, \s_{n-1} \rangle) \times 
	(\langle \s_1', \ldots, \s_{n-2}' \rangle \cap \langle \s_2', \ldots, \s_{n-1}' \rangle) \\
	&= \langle \s_2, \ldots, \s_{n-2} \rangle \times \langle \s_2', \ldots, \s_{n-2}' \rangle,
	\end{align*}
	where the last line follows from the polytopality of $\calP$ and $\calQ$.
	Now, $\langle \s_2, \ldots, \s_{n-2} \rangle$ is the group of a polytope of type $\{p_2, \ldots, p_{n-2}\}$,
	and $\langle \s_2', \ldots, \s_{n-2}' \rangle$ is the group of a polytope of type $\{q_2, \ldots, q_{n-2}\}$.
	Since $p_i$ and $q_i$ are coprime for $i=2, \ldots, n-2$, the mix of those two groups is their direct product.
	That is, 
	\[ \langle \beta_2, \ldots, \beta_{n-2} \rangle = \langle \s_2, \ldots, \s_{n-2} \rangle \times
	\langle \s_2', \ldots, \s_{n-2}' \rangle. \]
	Thus we see that $\calP \mix \calQ$ is polytopal.
	\end{proof}
	
	\begin{corollary}
	\label{cor:polyhedra}
	Let $\calP$ and $\calQ$ be chiral or directly regular polyhedra. Then $\calP \mix \calQ$ is a chiral
	or directly regular polyhedron.
	\end{corollary}

\section{Regular covers and the chirality group}

	The chirality group of a hypermap was introduced in \cite{bjns} and then adapted to polytopes in \cite{const}.
	Instead of a binary invariant, we now get a much more detailed measure of how far away a polytope
	is from being directly regular. Additionally, we can use the chirality group to easily determine when
	the mix of two polytopes is chiral, thereby dodging lengthy calculations.
	
	Let $\calP$ be a chiral polytope, and $\ch{\calP}$ its enantiomorphic form (mirror image). If $\GP$ has presentation
	\[ \gp{\s_1, \ldots, \s_{n-1}}{w_1, \ldots, w_t},\]
	then the group $\GcP$ has presentation
	\[ \gp{\s_1, \ldots, \s_{n-1}}{\ch{w_1}, \ldots, \ch{w_t}},\]
	where we obtain $\ch{w}$ from $w$ by changing every $\s_1$ to $\s_1^{-1}$ and every $\s_2$ to $\s_1^2 \s_2$. 
	Let $\GP = W^+/M$, so that $\GcP = W^+/\ch{M}$.
	
	Now, the group $\GP \mix \GcP$ is isomorphic to $W^+/(M \cap \ch{M})$. Set $N = M \cap \ch{M}$. Then
	\[ \ch{N} = \ch{M} \cap \ch{\ch{M}} = \ch{M} \cap M = N, \]
	so that $\GP \mix \GcP$ is the group of a directly regular pre-polytope. Furthermore, any directly
	regular pre-polytope that covers $\calP$ must cover $\calP \mix \ch{\calP}$, so $\calP \mix \ch{\calP}$
	is the minimal directly regular pre-polytope that covers $\calP$. In a similar way, we see that $\calP
	\comix \ch{\calP}$ is the maximal directly regular quotient of $\calP$.

	In order to determine how chiral $\calP$ is, we compare $\GP$ to $\GP \mix \GcP$ or to $\GP \comix \GcP$.
	From \pref{prop:SizeOfMix}, we know that the natural maps $\pi_1: \GP \mix \GcP \to \GcP$ and $\pi_2: 
	\GP \to \GP \comix \GcP$ have isomorphic kernels. We call this kernel the \emph{chirality group} of
	$\calP$ and denote it by $X(\calP)$. At one extreme, $X(\calP)$ might be trivial, in which case
	$\calP$ is directly regular. At the other extreme, $X(\calP)$ might coincide with the whole
	automorphism group $\GP$; in that case, we say that $\calP$ is \emph{totally chiral}.

	Now we move on to our main application of the chirality group. Our goal is to construct chiral polytopes
	via mixing. If $\calP$ is a chiral polytope and $\calQ$ is a chiral or directly regular polytope, how
	do we know when $\calP \mix \calQ$ is chiral? In principle, we can determine the answer with a computer
	algebra system. However, the usual algorithmic hurdles in combinatorial group theory often make it
	difficult to get an answer that way. 

	We start by giving an equivalent condition to direct regularity of $\calP \mix \calQ$.
	
	\begin{lemma}
	\label{lem:regular-mix-crit}
	Let $\calP$ and $\calQ$ be chiral or directly regular $n$-polytopes. Then
	$\calP \mix \calQ$ is directly regular if and only if $\GP \mix \GQ$ naturally covers $\GP \mix \GcP$
	and $\GQ \mix \GcQ$.
	\end{lemma}
	
	\begin{proof}
	Let $M$ and $K$ be the subgroups of $W^+$ such that $\GP = W^+/M$ and $\GQ = W^+/K$.
	Then $\GP \mix \GQ = W^+/(M \cap K)$. If $\calP \mix \calQ$ is directly regular, then we have that
	$M \cap K = \ch{M} \cap \ch{K}$. Then $M \cap K = M \cap K \cap \ch{M} \leq M \cap \ch{M}$, so
	$\GP \mix \GQ$ naturally covers $W^+/(M \cap \ch{M}) = \GP \mix \GcP$. Similarly, $\GP \mix \GQ$
	naturally covers $\GQ \mix \GcQ$. 
	
	Conversely, if $\GP \mix \GQ$ naturally covers both $\GP \mix \GcP$ and $\GQ \mix \GcQ$, then we
	have that $M \cap K \leq M \cap \ch{M}$ and $M \cap K \leq K \cap \ch{K}$. Therefore $M \cap K
	\leq M \cap \ch{M} \cap K \cap \ch{K}$, and thus we must have $M \cap K = \ch{M} \cap \ch{K}$,
	so that $\calP \mix \calQ$ is directly regular.
	\end{proof}

	We now come to the main result:
	
	\begin{theorem}
	\label{thm:chiral-mix-criterion}
	Let $\calP$ and $\calQ$ be finite chiral or directly regular polytopes, but not both directly regular. 
	Suppose that $|X(\calP)|$ does not divide $|\GQ|$ or that $|X(\calQ)|$ does not divide $|\GP|$. Then 
	$\calP \mix \calQ$ is chiral.
	\end{theorem}
	
	\begin{proof}
	Suppose that $\calP \mix \calQ$ is directly regular. Then by \lref{lem:regular-mix-crit}, $\GP \mix \GQ$ naturally covers $\GP \mix \GcP$ and $\GQ \mix \GcQ$. Thus we have that $|\GP \mix \GcP|$ and $|\GQ \mix \GcQ|$
	both divide $|\GP \mix \GQ|$, which divides $|\GP||\GQ|$. Since $|\GP \mix \GcP| = |\GP||X(\calP)|$ and
	$|\GQ \mix \GcQ| = |\GQ||X(\calQ)|$, we see that $|X(\calP)|$ divides $|\GQ|$ and $|X(\calQ)|$ divides
	$|\GP|$, and the result follows.
	\end{proof}
	
	\begin{corollary}
	Let $\calP$ and $\calQ$ be totally chiral polytopes. If $|\GP| \neq |\GQ|$, then $\calP \mix \calQ$ is chiral.
	\end{corollary}

	\tref{thm:chiral-mix-criterion} gives us a simple combinatorial criterion for the chirality of $\calP
	\mix \calQ$. In order to use it, we only need information about $\calP$ and $\calQ$ separately; no other
	information about $\calP \mix \calQ$ is required. Furthermore, finding the size of $X(\calP)$ is
	usually a simple computation. We just need to calculate $|\GP| / |\GP \comix \GcP|$.
	
	Let us consider a simple example to illustrate. Let $\calQ$ be a finite chiral or directly regular, polyhedron, and let $\calP$ be the chiral polyhedron $\{4, 4\}_{(b, c)}$, where $p := b^2 + c^2$ is an odd
	prime. Suppose $p > |\GQ|$. Since the chirality group of $\calP$ is cyclic of order $p$ \cite{cox-index},	
	\tref{thm:chiral-mix-criterion} tells us that the polyhedron $\calP \mix \calQ$ is chiral.
	
	A small problem arises when we try to apply \tref{thm:chiral-mix-criterion} multiple times. In order
	to determine if $(\calP \mix \calQ) \mix \calR$ is chiral, we need to know the chirality group of
	$\calP \mix \calQ$. All that \tref{thm:chiral-mix-criterion} tells us is whether $X(\calP \mix \calQ)$
	is trivial. However, with a little extra work, we can get a lower bound on $X(\calP \mix \calQ)$.
	
	\begin{theorem}
	\label{thm:chirality-gp-size}
	Let $\calP$ and $\calQ$ be finite chiral or directly regular polytopes. Let $g_1$ be the greatest common
	divisor of $|X(\calP)|$ and $|\GQ|$, and let $g_2$ be the greatest common divisor of $|X(\calQ)|$ and
	$|\GP|$. Then $|X(\calP \mix \calQ)|$ is divisible by $|X(\calP)| / g_1$ and by $|X(\calQ)| / g_2$.
	\end{theorem}

	\begin{proof}
	The group $\GP \mix \GcP$ is covered by $(\GP \mix \GQ) \mix (\GcP \mix \GcQ)$. The former has size
	$|\GP||X(\calP)|$, while the latter has size $|\GP \mix \GQ||X(\calP \mix \calQ)|$. Then $|\GP||X(\calP)|$
	divides $|\GP \mix \GQ||X(\calP \mix \calQ)|$, which divides $|\GP||\GQ||X(\calP \mix \calQ)|$.
	Therefore, $|X(\calP)|$ divides $|\GQ||X(\calP \mix \calQ)|$. Since $g_1$ divides both $|X(\calP)|$
	and $|\GQ|$, we get that $|X(\calP)| / g_1$ divides $|\GQ||X(\calP \mix \calQ)| / g_1$. Furthermore,
	$|X(\calP)| / g_1$ is coprime to $|\GQ| / g_1$, and thus $|X(\calP)| / g_1$ must divide $|X(\calP \mix \calQ)|$.
	The result then follows by symmetry.
	\end{proof}
	
	Armed with this new result, we can now say something about the chirality of $(\calP \mix \calQ) \mix \calR$.
	In particular, if $(\calP \mix \calQ) \mix \calR$ is directly regular, then $|X(\calP \mix \calQ)|$ divides
	$|\GR|$ by \tref{thm:chiral-mix-criterion}; by \tref{thm:chirality-gp-size}, we conclude that 
	$|X(\calP)| / g_1$ and $|X(\calQ)| / g_2$ both divide $|\GR|$.
	
	Let us return to our previous example, where $\calP$ is the chiral polyhedron $\{4, 4\}_{(b, c)}$ with a 
	chirality group of order $p = b^2 + c^2$. Suppose $p$ is an odd prime, and let $\calQ$ and $\calR$ be finite
	chiral or directly regular polyhedra such that $p$ does not divide $|\GQ|$ or $|\GR|$. Then by
	\tref{thm:chirality-gp-size}, $p$ must divide $|X(\calP \mix \calQ)|$. Since $p$ does not divide $|\GR|$,
	that means that $|X(\calP \mix \calQ)|$ does not divide $|\GR|$, and thus by \tref{thm:chiral-mix-criterion},
	the polyhedron $\calP \mix \calQ \mix \calR$ is chiral.
	
	If the chirality group of $\calP$ is simple (as it was in our previous example) and $\calP \mix \calQ$
	is chiral, we can actually determine the chirality group of $\calP \mix \calQ$. \pref{prop:chirality-subgroup}
	below was proved in \cite{const}.
	
	\begin{proposition}
	\label{prop:chirality-subgroup}
	Let $\calP$ be a chiral $n$-polytope and let $\calQ$ be a directly regular $n$-polytope. Then $X(\calP \mix
	\calQ)$ is a normal subgroup of $X(\calP)$.
	\end{proposition}
	
	\begin{theorem}
	\label{thm:simple-chirality-gp}
	Let $\calP$ be a finite chiral polytope such that $X(\calP)$ is simple. Let $\calQ$ be a finite directly
	regular polytope. If $|X(\calP)|$ does not divide $|\GQ|$, then $\calP \mix \calQ$ is chiral and
	$X(\calP \mix \calQ) = X(\calP)$.
	\end{theorem}
	
	\begin{proof}
	By \tref{thm:chiral-mix-criterion}, the mix $\calP \mix \calQ$ is chiral. \pref{prop:chirality-subgroup} says
	that $X(\calP \mix \calQ)$ is a normal subgroup of $X(\calP)$, which is simple. Since $\calP \mix
	\calQ$ is chiral, $X(\calP \mix \calQ)$ must be nontrivial, and thus we have that $X(\calP \mix \calQ) =
	X(\calP)$.
	\end{proof}
	
	Thus, we see that if $\calP$ is a chiral polytope with simple chirality group, and if $\calQ_1, \ldots,
	\calQ_k$ are finite directly regular polytopes such that $|X(\calP)|$ does not divide any $|\G^+(\calQ_i)|$, 
	then $\calP \mix \calQ_1 \mix \cdots \mix \calQ_k$ is chiral.
	
	Returning to the mix of two polytopes, we also get a nice result in the case where $\calQ$ has a simple
	rotation group:
	
	\begin{theorem}
	\label{thm:simple-rot-gp}
	Let $\calP$ be a finite chiral polytope, and let $\calQ$ be a finite chiral polytope such 
	that $\GQ$ is simple. If $X(\calP)$ is not isomorphic to $\GQ$, then $\calP \mix \calQ$ is chiral.
	\end{theorem}
	
	\begin{proof}
	Suppose $\calP \mix \calQ$ is directly regular. Then by \lref{lem:regular-mix-crit}, $\GP \mix \GQ$ naturally
	covers $\GP \mix \GcP$. Thus, we get the following commutative diagram, where the maps are all the natural
	covering maps:
	\[ \xymatrix{
	\GP \mix \GQ \ar[rr]^{f_1} \ar[rd]_{f_2} & & \GP \mix \GcP \ar[dl]^{g} \\
	& \GP &
	} \]
	So we have that $\ker f_2 = \ker (g \circ f_1)$, and in particular, $\ker f_1$ is normal
	in $\ker f_2$. Now, we can view $\ker f_2$ as a normal subgroup of $\GQ$, which is simple. 
	Then $\ker f_2$ is simple (possibly trivial), and $\ker f_1$ is likewise simple (possibly trivial).
	Suppose $\ker f_1$ is trivial. Then $\ker f_2 = \ker g = X(\calP)$. Since $\calP$ is chiral, $\ker g$
	is nontrivial. Therefore, $\ker f_2$ is nontrivial, so we must have $\ker f_2 \simeq \GQ$. But then
	$X(\calP) \simeq \GQ$, violating our assumptions. Now, suppose instead that $\ker f_1 \simeq \GQ$. Then we
	must have $\ker f_2 \simeq \GQ$ as well. Then since $\ker f_2 = \ker (g \circ f_1)$, we must have
	that $\ker g$ is trivial, contradicting that $\calP$ is chiral. Therefore, $\calP \mix \calQ$ must
	be chiral.
	\end{proof}
	
	\begin{theorem}
	\label{thm:simple-rot-gp-2}
	Let $\calP$ be a finite chiral polytope, and let $\calQ$ be a finite directly regular polytope with a simple
	rotation group $\GQ$. If $X(\calP)$ is not isomorphic to $\GQ$, then $X(\calP \mix \calQ) = X(\calP)$.
	\end{theorem}
	
	\begin{proof}
	Consider the commutative diagram below, where the maps are all the natural covering maps:
	\[ \xymatrix@!C{
	\G^+(\calP \mix \ch{\calP} \mix \calQ \mix \ch{\calQ}) \ar[r]^{f_1} \ar[d]_{f_2} &
	\G^+(\calP \mix \ch{\calP}) \ar[d]^{g_1} \\
	\G^+(\calP \mix \calQ) \ar[r]_{g_2} & \GP
	} \]
	We have that $|\ker f_2|\cdot|\ker g_2| = |\ker f_1|\cdot|\ker g_1|$. Furthermore, $\ker f_2 =
	X(\calP \mix \calQ)$ and $\ker g_1 = X(\calP)$. Now, since $\calQ$ is directly regular, the group
	$\G^+(\calP \mix \ch{\calP} \mix \calQ \mix \ch{\calQ})$ is the same as $\G^+(\calP \mix \ch{\calP} \mix 
	\calQ)$, and therefore, we can view $\ker f_1$ as a normal subgroup of $\GQ$. We can similarly
	view $\ker g_2$ as a normal subgroup of $\GQ$. Suppose $\ker g_2$ is trivial. Then $\calP \mix \calQ
	= \calP$, so that $X(\calP \mix \calQ) = X(\calP)$. Otherwise, we must have that $\ker g_2 \simeq
	\GQ$ since $\GQ$ is simple. Then we get that 
	\[ |X(\calP \mix \calQ)|\cdot|\GQ| = |\ker f_1|\cdot|X(\calP)| \leq |\GQ|\cdot|X(\calP)|. \]
	Therefore, $|X(\calP \mix \calQ)| \leq |X(\calP)|$. By \pref{prop:chirality-subgroup}, $X(\calP)$
	is a normal subgroup of $X(\calP \mix \calQ)$, so we see that again we must have $X(\calP) = X(\calP \mix
	\calQ)$.
	\end{proof}

	Several recent papers address which finite simple groups are the automorphism groups of regular polytopes
	(for example, see \cite{leemans-simple}, \cite{hartley-simple}).
	By mixing these regular polytopes with chiral polytopes, we get a rich source of new chiral polytopes.

	So far, we have only considered the chirality groups of finite polytopes. However, a claim similar to
	\tref{thm:chirality-gp-size} is true for chiral polytopes with an infinite chirality group:
	
	\begin{theorem}
	\label{thm:infinite-chirality-gp}
	Let $\calP$ be a chiral polytope such that $X(\calP)$ is infinite. Let $\calQ$ be a finite chiral or directly 
	regular polytope. Then $X(\calP \mix \calQ)$ is infinite.
	\end{theorem}
	
	\begin{proof}
	Consider the commutative diagram below, where the maps are all the natural covering maps:
	\[ \xymatrix@!C{
	\G^+(\calP \mix \ch{\calP} \mix \calQ \mix \ch{\calQ}) \ar[r]^{f_1} \ar[d]_{f_2} &
	\G^+(\calP \mix \ch{\calP}) \ar[d]^{g_1} \\
	\G^+(\calP \mix \calQ) \ar[r]_{g_2} & \GP
	} \]
	Then $\ker (g_1 \circ f_1) = \ker (g_2 \circ f_2)$. Now, since $X(\calP) = \ker g_1$ is infinite,
	it follows that $\ker (g_1 \circ f_1)$ is infinite. Now, $\ker f_2 = X(\calP \mix \calQ)$, and we can view 
	$\ker g_2$ as a subgroup of $\GQ$, which is finite. If $X(\calP \mix \calQ)$ is finite, then so is 
	$\ker (g_2 \circ f_2) = \ker (g_1 \circ f_1)$.  So we must have that $X(\calP \mix \calQ)$ is infinite.
	\end{proof}
	
	Determining whether $\calP \mix \calQ$ is chiral when both $\calP$ and $\calQ$ are infinite is sometimes
	possible, but it requires more machinery, and we won't develop it here.
		
\section{New examples of chiral polytopes}

	Now we will apply the results of the previous section in order to construct new examples of chiral polytopes. 
	We start by giving a few concrete new examples of small chiral $5$-polytopes. 
	
	Let $\calP$ be the chiral polytope of type $\{3, 4, 4, 3\}$ in \cite{chp}, with automorphism group $S_6$.
	A direct calculation in GAP \cite{gap} shows that the chirality group is $A_6$. Let $\calQ$ be the
	degenerate directly regular polytope $\{2, 3, 3, 2\}$. Then we have that $|X(\calP)| = 360$, while
	$|\GQ| = 48$. Then by \tref{thm:simple-chirality-gp}, $\calP \mix \calQ$ is chiral with chirality
	group $A_6$. Furthermore, by \pref{prop:rel-prime-type}, $\calP \mix \calQ$ is polytopal. We get a chiral
	polytope of type $\{6, 12, 12, 6\}$ with 12 vertices, 120 edges, 480 2-faces, 120 3-faces, and 12 4-faces.
	The total number of flags is 69120.
	
	Mixing $\calP$ with $\{2, 3, 3, 3\}$ also yields a new chiral polytope. In this case, \tref{thm:rel-prime-type}
	tells us that the mix is polytopal. We get a chiral polytope of type $\{6, 12, 12, 3\}$
	with 12 vertices, 150 edges, 2400 2-faces, 300 3-faces, and 30 4-faces. The number of flags is
	1728000.
	
	We now turn to some more general results. One particularly useful source of directly regular polytopes to mix 
	with chiral polytopes are the cubic $n$-toroids
	$\{4, 3^{n-3}, 4\}_{(s^k, 0^{n-k-1})}$, defined in \cite[Ch. 6D]{arp}. For each $s \geq 2$ and $k = 1,
	2,$ or $n-1$, we get a directly regular polytope with a group of size $2^{n+k-2} (n-1)! s^{n-1}$.
	Mixing them with chiral polytopes yields the following:
	
	\begin{theorem}
	\label{thm:toroid-mixing}
	Let $\calP$ be a finite chiral $n$-polytope $(n \geq 3)$, and let $\calQ(s,k)$ be the directly regular $n$-
	polytope $\{4, 3^{n-3}, 4\}_{(s^k, 0^{n-k-1})}$, where $k = 1, 2,$ or $n-1$. If $|X(\calP)|$ does not
	divide $2^{n+k-2} (n-1)!$, then there are infinitely many values of $s$ such that
	$\calP \mix \calQ(s,k)$ is chiral.
	\end{theorem}
	
	\begin{proof}
	If $s$ is a prime that does not divide $|X(\calP)|$, then $|X(\calP)|$ does not divide 
	$|\G^+(\calQ(s,k))| = 2^{n+k-2} (n-1)! s^{n-1}$. Therefore, by \tref{thm:chiral-mix-criterion},
	$\calP \mix \calQ(s,k)$ is chiral for each such $s$.
	\end{proof}

	The condition that $|X(\calP)|$ does not divide $2^{n+k-2} (n-1)!$ is very mild. For example, this holds
	if $|X(\calP)|$ has a prime factor larger than $n-1$. Note, however, that \tref{thm:toroid-mixing} says
	nothing about the polytopality of $\calP \mix \calQ(s, k)$ -- we must establish that separately.

	Next, we showcase one nice example of \tref{thm:simple-rot-gp-2}. Recall that the rotation group of the
	$n$-simplex is $A_{n+1}$, which is simple when $n \geq 4$.
	
	\begin{theorem}
	Let $\calP$ be a finite chiral $n$-polytope ($n \geq 4$) such that $X(\calP)$ is not isomorphic to
	$A_{n+1}$. Let $\calQ$ be the $n$-simplex. Then $\calP \mix \calQ$ is chiral, and $X(\calP \mix \calQ) =
	X(\calP)$.
	\end{theorem}
	
	Now, let us look at an example using \tref{thm:infinite-chirality-gp}. In \cite{chiral-ext}, the
	authors proved that for any chiral polytope $\calK$ with directly regular facets, there is a
	universal chiral polytope $U(\calK)$ with facets isomorphic to $\calK$. This polytope covers all
	other chiral polytopes with facets isomorphic to $\calK$. In many cases, this extension
	will have an infinite chirality group:
	
	\begin{lemma}
	\label{lem:ch-gp-of-extension}
	Let $\calK$ be a chiral polytope with directly regular facets such that $\GK \comix \G^+(\ch{\calK})$ is 
	finite. (For example, $\GK$ might itself be finite.) Let $\calP = U(\calK)$. Let $\GP = \langle \s_1, \ldots,
	\s_n \rangle$ so that $\GK = \langle \s_1, \ldots, \s_{n-1} \rangle$. If $\s_{n-1} = \eps$ in
	$\GK \comix \G^+(\ch{\calK})$, or if $\s_{n-1} = \s_{n-2}$ in $\GK \comix \G^+(\ch{\calK})$,
	then $X(\calP)$ is infinite.
	\end{lemma}
	
	\begin{proof}
	Since $X(\calP)$ is the kernel of the natural covering from $\GP$ to $\GP \comix \GcP$, it suffices to
	show that $\GP \comix \GcP$ is finite. Now, suppose that $\s_{n-1} = \eps$ in $\GK \comix \G^+(\ch{\calK})$.
	Then $\s_{n-1} = \eps$ in $\GP \comix \GcP$. Furthermore, the relation $(\s_{n-1} \s_n)^2 = \eps$ also holds
	in $\GP \comix \GcP$, so we get that $\s_n^2 = \eps$. Suppose instead that $\s_{n-1} = \s_{n-2}$ in
	$\GK \comix \G^+(\ch{\calK})$ (and thus in $\GP \comix \GcP$). We have that $(\s_{n-2} \s_{n-1} \s_n)^2
	= \eps$, and thus $(\s_{n-1}^2 \s_n)^2 = \eps = (\s_{n-1} \s_n)^2$. Therefore, $\s_{n-1} \s_n \s_{n-1}
	= \s_n$. So we get that $(\s_{n-1} \s_n)^2 = \s_n^2 = \eps$.

	Now, $\s_n$ commutes with $\s_i$ if $i < n-2$ \cite{SW1}. Using the relations $(\s_{n-1} \s_n)^2 = \eps$
	and $(\s_{n-2} \s_{n-1} \s_n)^2 = \eps$, we can conclude that $\s_n \s_{n-1} = \s_{n-1}^{-1} \s_n$ and that 
	$\s_n \s_{n-2} = \s_{n-2} \s_{n-1}^2 \s_n$. Therefore, given any word $w$ in $\GP \comix \GcP$, we can bring
	every $\s_n$ to the right and write $w = u \s_n^k$, where $u$ is a word in $\GK \comix \G^+(\ch{\calK})$.
	Therefore, $\GP \comix \GcP$ is finite; in particular, it is at most twice as large as
	$\GK \comix \G^+(\ch{\calK})$.
	\end{proof}

	\begin{theorem}
	\label{thm:infinite-extensions}
	Let $\calK$ be a finite chiral polytope with directly regular facets and group $\GK = \langle \s_1, \ldots,
	\s_{n-1} \rangle$. Let $\calP = U(\calK)$, and let $\calQ$ be a finite directly regular polytope.
	If $\s_{n-1} = \eps$ or $\s_{n-1} = \s_{n-2}$ in $\GK \comix \G^+(\ch{\calK})$, then $\calP \mix
	\calQ$ is chiral, with an infinite chirality group.
	\end{theorem}
	
	\begin{proof}
	By \lref{lem:ch-gp-of-extension}, the conditions given on $\calK$ suffice to ensure that $\calP$ has
	an infinite chirality group. Then \tref{thm:infinite-chirality-gp} applies to show that $\calP \mix
	\calQ$ also has an infinite chirality group.
	\end{proof}

	For example, let $\calK$ be the chiral polyhedron $\{4, 4\}_{(b, c)}$ with $p = b^2 + c^2$ an odd prime.
	Then $\GK \comix \G^+(\ch{\calK}) = [4, 4]_{(1, 0)}$, in which $\s_1 = \s_2$ \cite{cox-index}.
	Therefore, \tref{thm:infinite-extensions} says that we can mix $U(\calK)$ with any finite directly
	regular polytope to obtain a polytope with infinite chirality group.
	
	Note that in \tref{thm:infinite-extensions}, we had to choose $\calK$ to have directly regular facets.
	Since the $(n-2)$ faces of a chiral polytope must be directly regular, there is no universal polytope
	$U(\calP)$ if $\calP$ has chiral facets. In particular, there is no universal polytope $U(U(\calK))$.
	However, using mixing, we can build a chiral polytope with directly regular facets $\calK \mix \ch{\calK}$:
	
	\begin{theorem}
	\label{thm:pseudo-extension}
	Let $\calK$ be a finite totally chiral polytope of type $\{p_1, \ldots, p_{n-2}\}$ with directly regular facets. Let $\calP = U(\calK)$, and let $\calQ$ be the finite directly regular polytope of type $\{p_1, 
	\ldots, p_{n-2}, 2\}$ with facets isomorphic to $\calK \mix \ch{\calK}$. Then $\calP \mix \calQ$ is a chiral
	polytope with directly regular facets.
	\end{theorem}
	
	\begin{proof}
	Since $\calK$ is totally chiral, $\GK \comix \G^+(\ch{\calK})$ is trivial. Then by 
	\tref{thm:infinite-extensions}, the mix $\calP \mix \calQ$ is chiral. Furthermore, the facets of
	$\calQ$ cover the facets of $\calP$, so by \pref{prop:facets-cover}, the mix is polytopal.
	\end{proof}
	
	Unfortunately, \tref{thm:pseudo-extension} cannot be repeatedly applied, because it requires that we
	start with a finite chiral polytope, and it produces an infinite chiral polytope. Perhaps with some more
	careful analysis, a similar result may be obtained for infinite chiral polytopes.
	
	Finally, we note that \lref{lem:ch-gp-of-extension}, \tref{thm:infinite-extensions}, and
	\tref{thm:pseudo-extension} all still apply when $\calP$ is any infinite chiral polytope with facets 
	isomorphic to $\calK$ -- the universality of $U(\calK)$ is not necessary.
	
	\subsection{Acknowledgements}
	
	I would like to thank E.~Schulte for helpful comments and discussions.

\end{document}